\documentclass[12pt]{amsart}
\usepackage{euler, amsfonts, amssymb, latexsym, epsfig,epic}

\usepackage{amscd,amssymb}
\usepackage{verbatim}
\usepackage{pstricks}
\usepackage{pst-node}
\usepackage{fullpage}

\input{diagrams}

\usepackage{bibentry}
\usepackage{amsmath}
\usepackage{amsthm}
\usepackage{amssymb}
\usepackage{amsfonts}
\usepackage{amsxtra}
\usepackage{amscd}
\usepackage{epsfig}
\usepackage{verbatim}

\usepackage{latexsym,amstext,epsfig}
\usepackage[all, knot]{xy}
\xyoption{arc}

\newsymbol\pp 1275
\newcommand{\Hom}{\operatorname{Hom}}
\newcommand{\End}{\operatorname{End}}
\newcommand{\Ext}{\operatorname{Ext}}

\newcommand{\rep}{\operatorname{rep}}

\newcommand{\GL}{\operatorname{GL}}

\newcommand{\ZZ}{\mathbb Z}

\newcommand{\RR}{\mathbb R}

\newcommand{\Mat}{\operatorname{Mat}}

\newcommand{\rel}{\operatorname{relint}}

\newcommand{\filt}{\operatorname{filt}}

\newcommand{\ddim}{\operatorname{\mathbf{dim}}}
\newcommand{\dd}{\operatorname{\mathbf{d}}}
\newcommand{\ee}{\operatorname{\mathbf{e}}}
\newcommand{\hh}{\operatorname{\mathbf{h}}}

\newcommand{\Eff}{\operatorname{Eff}}
\newcommand{\module}{\operatorname{mod}}

\newtheorem{theorem}{Theorem}[section]
\newtheorem{proposition}[theorem]{Proposition}
\newtheorem{corollary}[theorem]{Corollary}
\newtheorem{lemma}[theorem]{Lemma}

\theoremstyle{definition}

\newtheorem{remark}[theorem]{Remark}
\newcount\cols
{\catcode`,=\active\catcode`|=\active
\gdef\Young(#1){\hbox{$\vcenter
{\mathcode`,="8000\mathcode`|="8000
\def,{\global\advance\cols by 1 &}%
\def|{\cr
      \multispan{\the\cols}\hrulefill\cr
       &\global\cols=2 }%
  \offinterlineskip\everycr{}\tabskip=0pt
  \dimen0=\ht\strutbox \advance\dimen0 by \dp\strutbox
    \halign
    {\vrule height \ht\strutbox depth \dp\strutbox##
      &&\hbox to \dimen0{\hss$##$\hss}\vrule\cr
     \noalign{\hrule}&\global\cols=2 #1\crcr
     \multispan{\the\cols}\hrulefill\cr%
   }
}$}} }

\begin{document}

\pagestyle{plain}

\mbox{}
\title{On the geometry of orbit closures for representation-infinite algebras}
\author{Calin Chindris}

\address{University of Missouri, Mathematics Department, Columbia, MO 65211, USA}
\email[Calin Chindris]{chindrisc@missouri.edu}

%\markboth{b}{b}
\date{August 5, 2011; Revised: \today}

\bibliographystyle{plain}
\maketitle

\begin{abstract} For the Kronecker algebra, Zwara found in \cite{Zwa3} an example of a module whose orbit closure is neither unibranch nor Cohen-Macaulay. In this paper, we explain how to extend this example to all representation-infinite algebras with a preprojective component. %In particular, this answers a question raised by Zwara in \cite{Zwa4}.
\end{abstract}

\section{Introduction}\label{intro-sec}

Throughout this paper, we work over an algebraically closed field $k$ of characteristic zero. All algebras (associative and with identity) are assumed to be finite-dimensional over $k$, and all modules are assumed to be finite-dimensional left modules.

One important problem in the geometric representation theory of algebras is that of describing the orbit closures of modules in module varieties. In \cite[Remark 5.1]{Zwa4}, Zwara asked wether the orbit closure of an arbitrary module over a tame concealed algebra is a unibranch variety. On the other hand, Zwara constructed in \cite[Theorem 1]{Zwa3} an orbit closure of a module over the Kronecker algebra with bad singularities. 

Our goal in this short paper is to explain how orthogonal exceptional sequences can be used to extend the aforementioned example to all connecetd representation-infinite algebras with a preprojective component (in particular, to all tame concealed algebras). We prove that:

\begin{theorem}\label{main:thm} Let $A=kQ/I$ be  a connected representation-infinite algebra with a preprojective component. Then, there exists a dimension vector $\dd \in \ZZ^{Q_0}_{\geq 0}$ and a module $M \in \module(A,\dd)$ such that the orbit closure $\overline{\GL(\dd)M}$ is neither unibranch  nor Cohen-Macaulay. \end{theorem}

In \cite[Corollary 1.3]{Zwa4}, Zwara showed that the orbit closures of modules for representation-finite algebras are always unibranch varieties. This result combined with Theorem \ref{main:thm} shows that an algebra with a preprojective component is representation-finite if and only if all of its orbit closures are unibranch. 

The layout of this paper is as follows. In Section \ref{module-varieties-sec}, we review background material on module varieties; in particular, we recall the notions of orthogonal exceptional sequences, and effective weights for finite-dimensional algebras. We prove Theorem \ref{main:thm} in Section \ref{proof-of-theorem-sec}.

\section{Background on module varieties} \label{module-varieties-sec} Let $Q=(Q_0,Q_1,t,h)$ be a finite quiver with vertex set $Q_0$ and arrow set $Q_1$. The two functions $t,h:Q_1 \to Q_0$ assign to each arrow $a \in Q_1$ its tail \emph{ta} and head \emph{ha}, respectively.

A representation $M$ of $Q$ over $k$ is a collection $(M(i),M(a))_{i\in Q_0, a\in Q_1}$ of finite-dimensional $k$-vector spaces $M(i)$, $i \in Q_0$, and $k$-linear maps $M(a) \in \Hom_k(M(ta), M(ha))$, $a \in Q_1$. The dimension vector of a representation $M$ of $Q$ is the function $\ddim M : Q_0 \to \ZZ$ defined by $(\ddim M)(i)=\dim_{k} M(i)$ for $i\in Q_0$. Let $S_i$ be the one-dimensional representation of $Q$ at vertex $i \in Q_0$ and let us denote by $\ee_i$ its dimension vector. By a dimension vector of $Q$, we simply mean a function $\dd \in \ZZ_{\geq 0}^{Q_0}$.

Given two representations $M$ and $N$ of $Q$, we define a morphism $\varphi: M \rightarrow N$ to be a collection $(\varphi(i))_{i \in Q_0}$ of $k$-linear maps with $\varphi(i) \in \Hom_k(M(i), N(i))$ for each $i \in Q_0$, and such that $\varphi(ha)M(a)=N(a)\varphi(ta)$ for each $a \in Q_1$. We denote by $\Hom_Q(M, N)$ the $k$-vector space of all morphisms from $M$ to $N$. Let $M$ and $N$ be two representations of $Q$. We say that $M$ is a subrepresentation of $N$ if $M(i)$ is a subspace of $N(i)$ for each $i \in Q_0$ and $M(a)$ is the restriction of $N(a)$ to $M(ta)$ for each $a \in Q_1$. In this way, we obtain the abelian category $\rep(Q)$ of all representations of $Q$.

Given a quiver $Q$, its path algebra $kQ$ has a $k$-basis consisting of all paths (including the trivial ones) and the multiplication in $kQ$ is given by concatenation of paths. It is easy to see that any $kQ$-module defines a representation of $Q$, and vice-versa. Furthermore, the category $\module(kQ)$ of $kQ$-modules is equivalent to the category $\rep(Q)$. In what follows, we identify $\module(kQ)$ and $\rep(Q)$, and use the same notation for a module and the corresponding representation.

A two-sided ideal $I$ of $kQ$ is said to be \emph{admissible} if there exists an integer $L\geq 2$ such that $R_Q^L\subseteq I \subseteq R_Q^2$. Here, $R_Q$ denotes the two-sided ideal of $kQ$ generated by all arrows of $Q$. 

If $I$ is an admissible ideal of $KQ$, the pair $(Q,I)$ is called a \emph{bound quiver} and the quotient algebra $kQ/I$ is called the \emph{bound quiver algebra} of $(Q,I)$.  It is well-known that any basic algebra $A$ is isomorphic to the bound quiver algebra of a bound quiver $(Q_{A},I)$, where $Q_{A}$ is the Gabriel quiver of $A$ (see \cite{AS-SI-SK}). (Note that the ideal of relations $I$ is not uniquely determined by $A$.) We say that $A$ is a \emph{triangular} algebra if its Gabriel quiver has no oriented cycles.

Fix a bound quiver $(Q,I)$ and let $A=kQ/I$ be its bound quiver algebra.  We denote by $e_i$ the primitive idempotent corresponding to the vertex $i\in Q_0$.  A representation $M$ of a $A$ (or  $(Q,I)$) is just a representation $M$ of $Q$ such that $M(r)=0$ for all $r \in I$. The category $\module(A)$ of finite-dimensional left $A$-modules is equivalent to the category $\rep(A)$ of representations of $A$. As before, we identify $\module(A)$ and $\rep(A)$, and make no distinction between $A$-modules and representations of $A$. 

Assume form now on that $A$ has finite global dimension; this happens, for example, when $Q$ has no oriented cycles. The Ringel form of $A$ is the bilinear form  $\langle \cdot, \cdot \rangle_{A} : \ZZ^{Q_0}\times \ZZ^{Q_0} \to \ZZ$ defined by
$$
\langle \dd,\ee \rangle_{A}=\sum_{l\geq 0}(-1)^l \sum_{i,j\in Q_0}\dim_k \Ext^l_{A}(S_i,S_j)\dd(i)\ee(j).
$$
Note that if $M$ is a $\dd$-dimensional $A$-module and $N$ is an $\ee$-dimensional $A$-module then
$$
\langle \dd,\ee \rangle_{A}=\sum_{l\geq 0}(-1)^l \dim_k \Ext^l_{A}(M,N).
$$
The quadratic form induced by $\langle \cdot,\cdot \rangle_{A}$ is denoted by $\chi_{A}$.

The \emph{Tits form} of $A$ is the integral quadratic form $q_{A}: \ZZ^{Q_0} \to \ZZ$ defined by
$$q_{A}(\dd):=\sum_{i \in Q_0}\dd^2(i)-\sum_{i,j\in Q_0}\dim_{k}\Ext^1_{A}(S_i,S_j)\dd(i)\dd(j)+\sum_{i,j\in Q_0}\dim_{k}\Ext^2_{A}(S_i,S_j)\dd(i)\dd(j).$$

If $A$ is triangular then $r(i,j):=|R \cap e_j\langle R \rangle e_i|$ is precisely $\dim_{k}\Ext^2_{A}(S_i,S_j)$, $\forall i,j \in Q_0$, as shown by Bongartz in \cite{Bon}. So, in the triangular case, we can write
$$
q_{A}(\dd)=\sum_{i \in Q_0}\dd^2(i)-\sum_{a\in Q_1}\dd(ta)\dd(ha)+\sum_{i,j\in Q_0}r(i,j)\dd(i)\dd(j).
$$

A dimension vector $\dd$ of $A$ is called a \emph{root} if $\dd$ is the dimension vector of an indecomposable $A$-module. A root $\dd$ of $A$ is said to be \emph{isotropic} if $q_{A}(\dd)=0$; we say it is \emph{real} if $q_{A}(\dd)=1$. Finally, we say that $\dd$ is a \emph{Schur root} if $\dd$ is the dimension vector of an $A$-module $M$ for which $\End_{A}(M) \simeq k$. Such a module $M$ is called a \emph{Schur module}.

Let $\dd$ be a dimension vector of $A$ (or equivalently, of $Q$). The variety of $\dd$-dimensional $A$-modules is the affine variety
$$
\module(A,\dd)=\{M \in \prod_{a \in Q_1} \Mat_{\dd(ha)\times \dd(ta)}(k) \mid M(r)=0, \forall r \in
I \}.
$$
It is clear that $\module(A,\dd)$ is a $\GL(\dd)$-invariant closed subset of the affine space $\module(Q,\dd):= \prod_{a \in Q_1} \Mat_{\dd(ha)\times \dd(ta)}(k)$. Note that $\module(A, \dd)$ does not have to be irreducible. We call $\module(A,\dd)$ the \emph{module variety} of $\dd$-dimensional $A$-modules. 

\subsection{Orthogonal exceptional sequences}\label{ortho-except-seq-sec} Recall that a sequence $\mathcal{E}=(E_1, \dots, E_t)$ of finite-dimensional $A$-modules is called an \emph{orthogonal exceptional sequence} if the following conditions are satisfied:
\begin{enumerate}
\renewcommand{\theenumi}{\arabic{enumi}}
\item $E_i$ is an exceptional module, i.e, $\End_A(E_i)=k$ and $\Ext^l_A(E_i,E_i)=0$ for all $l \geq 1$ and $1 \leq i \leq t$;

\item $\Ext_A^l(E_i,E_j)=0$ for all $l \geq 0$  and $1 \leq i<j \leq t$;

\item $\Hom_A(E_j,E_i)=0$ for all $1 \leq i<j \leq t$.
\end{enumerate}
(If we drop condition $(3)$, we simply call $\mathcal{E}$ an \emph{exceptional sequence}.)

Given an orthogonal exceptional sequence $\mathcal{E}$, consider the full subcategory $\filt_{\mathcal{E}}$ of $\module(A)$ whose objects $M$ have a finite filtration $0=M_0\subseteq M_1 \subseteq \dots \subseteq M_s=M$ of submodules such that each factor $M_j/M_{j-1}$ is isomorphic to one the $E_1, \ldots, E_t$. It is clear that $\filt_{\mathcal{E}}$ is a full exact subcategory of $\module(A)$ which is closed under extensions. Moreover, Ringel \cite{R2} (see also \cite{DW2}) showed that $\filt_{\mathcal{E}}$ is an abelian subcategory whose simple objects are precisely $E_1, \ldots, E_t$.

Let $A_{\mathcal{E}}=kQ_{\mathcal{E}}/I_{\mathcal{E}}$ be the bound quiver algebra where the Gabriel quiver $Q_{\mathcal{E}}$ has vertex set $\{1, \ldots, t\}$ and $\dim_k \Ext_A^1(E_i,E_j)$ arrows from $i$ to $j$ for all $1\leq i, j \leq t$. The ideal $I_{\mathcal{E}}$ is determined by the $A_{\infty}$-algebra structure of the the Yoneda algebra $\Ext^{\bullet}_Q(\bigoplus_{i=1}^t E_i, \bigoplus_{i=1}^t E_i)$. From the work of Keller \cite{Kel1, Kel2}, we know that there exists an equivalence of categories $F_{\mathcal{E}}: \module(A_{\mathcal{E}}) \to \filt_{\mathcal{E}}$ sending the simple $A_{\mathcal{E}}$-module $S_i$ at vertex $i$ to $E_i$ for all $1 \leq i \leq t$. 

Now, consider a dimension vector $\dd'$ of $Q_{\mathcal{E}}$ and set $\dd=\sum_{1 \leq i \leq t}\dd'(i)\ddim E_i$. Then, there exist a regular morphism $f_{\mathcal{E}}:\module(A_{\mathcal{E}},\dd') \to \module(A, \dd)$ such that $f_{\mathcal{E}}(M') \simeq F_{\mathcal{E}}(M')$ for all $M' \in \module(A_{\mathcal{E}},\dd')$ (for more details, see \cite[Section 5]{CC9}).

As an immediate consequence of Zwara's Theorem 1.2 in \cite{Zwa5}, we have:

\begin{proposition}\label{exceptional-Zwara-prop} Keep the same notations as above and let $M' \in \module(A_{\mathcal{E}},\dd')$. Then, $\overline{\GL(\dd')M'}$ is smooth/unibranch/Cohen-Macaulay at some $N'$ if and only if the same is true for $\overline{\GL(\dd)f_{\mathcal{E}}(M')}$ at $f_{\mathcal{E}}(N')$.
\end{proposition}

\begin{remark} In particular, this proposition allows us to construct orbit closures in $\module(A,\dd)$ with bad singularities by reducing the considerations to the smaller algebra $A_{\mathcal{E}}$. What is needed at this point is an effective method for constructing convenient orthogonal exceptional sequences. This is addressed in the section below. 
\end{remark}

\subsection{Cones of effective weights}\label{Eff-sec} Let $\dd$ be a dimension vector of $A$ and let $\theta \in \RR^{Q_0}$ be a real weight. Given a vector $\dd' \in \RR^{Q_0}$, we define $\theta(\dd')=\sum_{i \in Q_0}\theta(i)\dd'(i)$. Recall that a module $M \in \module(A)$ is said to be \emph{$\theta$-semi-stable} if $\theta(\ddim M)=0$ and $\theta(\ddim M') \leq 0$ for all submodules $M' \subseteq M$. We say that $M$ is \emph{$\theta$-stable} if  $\theta(\ddim M)=0$ and $\theta(\ddim M') < 0$ for all proper submodules $\{0\} \subset M' \subset M$. Denote by $\module(A)^{ss}_{\theta}$ the full subcategory of $\module(A)$ consisting of all $\theta$-semi-stable $A$-modules. Then, $\module(A)^{ss}_{\theta}$ is an abelian subcategory of $\module(A)$ which is closed under extensions, and whose simple objects are precisely the $\theta$-stable $A$-modules. Moreover, $\module(A)^{ss}_{\theta}$ is Artinian and Noetherian, and hence, every $\theta$-semi-stable finite-dimensional $A$-module has a Jordan-H{\"o}lder filtration in $\module(A)^{ss}_{\theta}$.

Now, let $C$ be an irreducible component of $\module(A,\dd)$. We define $C^{s(s)}_{\theta}=\{M \in C \mid M \text{~is~} \theta\text{-(semi-)stable} \}$. The cone of effective weights of $C$ is, by definition, the set
$$
\Eff(C)=\{\theta \in \RR^{Q_0}\mid C^{ss}_{\theta}\neq \emptyset \}.
$$

It is well known that $\Eff(C)$ is a rational convex polyhedral cone of dimension at most $|Q_0|-1$. Given a lattice point $\theta_0$ in $\Eff(C)$, we say that $$\dd=\dd_1\pp \ldots \pp \dd_t$$ is the \emph{$\theta_0$-stable decomposition of $\dd$ in $C$} if the generic module $M$ in $C$ has a Jordan-H{\"o}lder filtration $\{0\}=M_0 \subset M_1 \subset \ldots \subset M_t=M$ in $\module(A)^{ss}_{\theta_0}$ such that the sequence $(\ddim M_1,\\ \ddim M_1/M_2, \ldots, \ddim M/M_{t-1})$ is the same as $(\dd_1,\ldots, \dd_t)$ up to permutation (for more details, see \cite[Section 6.2]{CC9}). If $\dd'$ is a dimension vector that occurs in a stable decomposition with multiplicity $m$, we write $m\cdot\dd'$ instead of $\underbrace{\dd' \pp \dd \pp \ldots \pp \dd'}_{m}$.

In what follows, we denote by $\mathbf H(\dd)$ the hyperplane in $\RR^{Q_0}$ orthogonal to a real-valued function $\dd \in \RR^{Q_0}$, i.e., $\mathbf H(\dd)=\{\theta \in \RR^{Q_0} \mid \theta(\dd)=0\}$.

\begin{lemma}\cite[Lemma 6.5]{CC9}\label{face-Eff-lemma} Let $\mathcal F$ be a face of $\Eff(C)$ of positive dimension, $\theta_0 \in \rel \Eff(C) \cap \ZZ^{Q_0}$, and 
$$\dd=m_1\cdot \dd_1 \pp \ldots m_t \cdot \dd_t$$ the $\theta_0$-stable decomposition of $\dd$ in $C$ with $\dd_i\neq \dd_j, \forall 1 \leq  i \neq j \leq t$. Then,

$$
\mathcal F=\Eff(C) \cap \bigcap_{i=1}^t \mathbf H(\dd_i).
$$
\end{lemma}

As a direct consequence of this lemma, we have the following useful result:

\begin{corollary}\label{facets-cor} Assume that $\Eff(C)$ has dimension $|Q_0|-1$ and let $\mathcal F$ be a facet of $\Eff(C)$. Let $\theta_0 \in \rel \Eff(C) \cap \ZZ^{Q_0}$ and let 
$$\dd=m_1\cdot \dd_1 \pp \ldots \pp m_t \cdot \dd_t$$ be the $\theta_0$-stable decomposition of $\dd$ in $C$ with $\dd_i\neq \dd_j, \forall 1 \leq  i \neq j \leq t$. 

If the dimension vectors $\dd_1, \ldots, \dd_t$ are indivisible then $\mathcal F=\Eff(\Lambda,\dd)\cap \mathbf{H}(\dd_1)\cap \mathbf{H}(\dd_2)$ and $\dd=n_1\dd_1+n_2\dd_2$ for unique numbers $n_1$ and $n_2$.
\end{corollary}

\begin{proof} Note that $\mathcal{F}$ has dimension $|Q_0|-2$, and so $t \geq 2$. Moreover, the dimension of the subspace of $\RR^{Q_0}$ orthogonal to the subspace spanned by $\{\dd, \dd_1,\dd_2\}$ is at least $|Q_0|-2$ since it contains $\mathcal F$. In particular, the set $\{\dd,\dd_1,\dd_2\}$ is linearly dependent. Since $\dd_1$ and $\dd_2$ are distinct indivisible vectors, we deduce that $\dd=n_1\dd_1+n_2\dd_2$ for unique numbers $n_1$ and $n_2$.

When $t=2$, the proof follows from Lemma \ref{face-Eff-lemma}. Now, let us assume that $t \geq 3$. Arguing as before, we deduce that $\dd$ is a linear combination of $\dd_i$ and $\dd_1$, and $\dd$ is also a linear combination of $\dd_i$ and $\dd_2$ for all $3 \leq i \leq t$. So, $\dd_i$ is a linear combination of  $\dd_1$ and $\dd_2$ for all $i$, and this implies that $\mathbf H(\dd_1) \cap \mathbf H(\dd_2)=\bigcap_{i=1}^t \mathbf H(\dd_i)$. The proof of the claim now follows again from Lemma \ref{face-Eff-lemma}.
\end{proof}

In the next section, we use this description of the facets of $\Eff(C)$ to prove the existence of short orthogonal exceptional sequences for tame concealed algebras. 

\section{Proof of Theorem \ref{main:thm}}\label{proof-of-theorem-sec}

We begin with the following example due to Zwara (see \cite{Zwa3}):

\begin{theorem}\label{Zwara-thm} Let $K_2$ be the Kronecker
quiver
$$
\xy     (0,0)*{1}="a";
        (10,0)*{2}="b";
        {\ar@2{->} "a";"b" };
\endxy
$$
Label the arrows by $a$ and $b$. Consider the following representation $M \in \rep(K_2,(3,3))$ defined by $M(a)=
\left(
\begin{matrix}
0& 0 & 0\\
1& 0 & 0\\
0& 1 & 0
\end{matrix}
\right)$ and $M(b)= \left(
\begin{matrix}
1& 0 & 0\\
0& 0 & 0\\
0& 0 & 1
\end{matrix}
\right)$. Then, $\overline{\GL((3,3))M}$ is neither unibranch nor Cohen-Macaulay.
\end{theorem}

It essentially follows from the work of Happel and Vossieck in \cite{HapVos} that a basic, connected, representation-infinite algebra admitting a preprojective component has a tame concealed algebra as a factor (see also \cite[Section XIV.3]{Sim-Sko-2}). Consequently, to prove our theorem, we can reduce the considerations to the tame concealed case. Let us now briefly recall some of the key features of a tame concealed algebra $A=kQ/I$. It is well-known that there is a unique indivisible dimension vector $\hh$ of $A$ such that $q_A(\hh)=0$. In fact, $\hh$ turns out to be the unique isotropic Schur root of $A$. Let $\theta_{\hh} \in \ZZ^{Q_0}$ be the weight defined by $\theta_{\hh}(\dd)=\langle \hh, \dd \rangle_A, \forall \dd \in \ZZ^{Q_0}$. Now,  let $\mathcal P$ ($\mathcal R, \mathcal Q$, respectively) be the full subcategory of $\module(A)$ consisting of all $A$-modules that are direct sums of indecomposable $A$-modules $X$ such that $\theta_{\hh}(\ddim X)<0$ ($\theta_{\hh}(\ddim X)=0, \theta_{\hh}(\ddim X)>0$, respectively). The following properties hold true.
\begin{enumerate}
\renewcommand{\theenumi}{\roman{enumi}}
\item $\module(A)=\mathcal P \bigvee \mathcal R \bigvee \mathcal Q$, where the symbol $\bigvee$ indicates the formation of the additive closure of the union of the subcategories involved.
\item $\Hom_A(N,M)=\Ext^1_A(M,N)=0$ if either $N \in \mathcal R \bigvee \mathcal Q, M \in \mathcal P$ or  $N \in \mathcal Q, M \in \mathcal P \bigvee \mathcal R$.

\item $pd_A M\leq 1$ for all $M \in \mathcal P \bigvee \mathcal R$ and $id_A N \leq 1$ for all $N \in \mathcal R \bigvee \mathcal Q$.
\end{enumerate}

The next two results have been proved for certain tame concealed algebras in \cite[Section 6.2]{CC9}, and the arguments in loc. cit. work for arbitrary tame concealed algebras. Nonetheless, for completeness and for the convenience of the reader, we provide the proofs below. 

\begin{lemma} \label{lemma-dim-Eff} If $A$ is a tame concealed algebra then $\module(A,\hh)^s_{\theta_{\hh}} \neq \emptyset$.
\end{lemma}

\begin{proof} First of all, it is clear that $\module(A,\hh)^{ss}_{\theta_{\hh}}\neq \emptyset$ since any $\hh$-dimensional $A$-module from $\mathcal R$ is $\theta_{\hh}$-semi-stable. Let $M\in \module(A,\hh)$ be an indecomposable module that lies in a homogeneous tube of $\mathcal R$. We are going to show that $M$ is $\theta_{\hh}$-stable. Assume to the contrary that $M$ is not $\theta_{\hh}$-stable and consider a Jordan-H{\"o}lder filtration of $M$ in $\module(A)^{ss}_{\theta_{\hh}}$. This way, we can see that $M$ must have a proper $\theta_{\hh}$-stable submodule $M'$. Then, $M'$ must belong to the homogeneous tube of $M$, and from this we deduce that $\ddim M'$ is an integer multiple of $\hh$. But this is a contradiction.
\end{proof}

\begin{proposition} \label{exceptional-seq-tubular-prop} If $A$ is a tame concealed algebra then there exists an orthogonal exceptional sequence $\mathcal{E}=(E_1,E_2)$ of $A$-modules such that $A_{\mathcal E}$ is the path algebra of the Kronecker quiver $K_2$.
\end{proposition}

\begin{proof} Let $\hh$ be the unique isotropic Schur root of $A$. The module variety $\module(A, \hh)$ is irreducible by Corollary 3 in \cite{BS1}, and let us denote its cone of effective weights by $\Eff(A, \hh)$. We know from Lemma \ref{lemma-dim-Eff} that there exists a module $M \in \module(A,\hh)$ which is $\theta_{\hh}$-stable. In other words, the subset $\Omega^0(M)$ of $\Eff(A, \hh)$, defined as $\Omega^0(M)=\{\theta \in \RR^{Q_0} \mid \theta(\hh)=0, \theta(\ddim M')<0, \forall \{0\}\subset M' \subset M\}$, is a non-empty open (with respect to the Euclidean topology) subset of $\mathbf{H}(\hh)$. We deduce from this that $\dim \Eff(A,\hh)=|Q_0|-1$. Next, choose a facet $\mathcal{F}$ of the cone $\Eff(A,\hh)$ and a weight $\theta_0\in \rel \mathcal F \cap \ZZ^{\Delta_0}$. Now, consider the $\theta_0$-stable decomposition of $\hh$ in $\module(A, \hh)$:
$$\hh=m_1\cdot \hh_1 \pp \ldots \pp m_t \cdot \hh_t,$$ with $m_1,\ldots, m_t$ positive integers and $\hh_i\neq \hh_j, \forall 1 \leq i \neq j \leq t$. Note that $\hh_1, \ldots, \hh_t$ are indivisible real Schur roots.

For each $1 \leq i \leq t$, let $E_i$ be a $\hh_i$-dimensional $\theta_0$-stable module that arises as a factor of a Jordan-H{\"o}lder filtration of a generic module $M$ in $\module(A, \hh)$. Note that we can choose $M$ to be $\theta_{\hh}$-stable by Lemma \ref{lemma-dim-Eff}. Furthermore, we have that $\Hom_A(E_i,E_j)=0, \forall 1 \leq i\neq j\leq t$, since $E_1, \ldots, E_t$ are pairwise non-isomorphic ($\theta_0$-)stable modules. 

Since  $\hh_1, \ldots, \hh_t$ are indivisible, we have $\mathcal F=\Eff(\Lambda,\hh)\cap \mathbf{H}(\hh_1)\cap \mathbf{H}(\hh_2)$ and $\hh=n_1\hh_1+n_2\hh_2$ for unique numbers $n_1$ and $n_2$ by Corollary \ref{facets-cor}.

We have that $q_A(\hh_1)=q_A(\hh_2)=1$, and $E_1$ and $E_2$ are exceptional $A$-modules. To simplify notation, set $l=-\langle \hh_1,\hh_2\rangle_A-\langle \hh_2,\hh_1\rangle_A$. Then, using the fact that $\hh$ is an isotropic root in the radical of $\chi_{A}$, we deduce that $2n_1=n_2l, 2n_2=n_1l$, and $n_1^2+n_2^2=ln_1n_2$ . From these relations and the fact that $\hh$ is indivisible, we deduce that $n_1=n_2=1$ and $l=2$. Without loss of generality, we can assume that $E_1$ is a submodule of $M$ and $E_2=M/E_1$. Then, we have that $\dim_k \Ext^1_{\Lambda}(E_2,E_1)>0$. 

In what follows, we show that $\mathcal{E}:=(E_1,E_2)$ is an orthogonal exceptional sequence with $\dim_k \Ext^1_{\Lambda}(E_2,E_1)=2$ and $\Ext^2_{\Lambda}(E_2,E_1)=0$. 

As $M$ is $\theta_{\hh}$-stable we have that $\theta_{\hh}(\hh_1)<0$ and $\theta_{\hh}(\hh_2)>0$. Using the properties (ii)-(iii) mentioned above, we conclude  that $\mathcal{E}$ is an orthogonal exceptional sequence; in particular, we have that $\langle \hh_1,\hh_2 \rangle_A =0$, and so $\dim_k \Ext^1_A(E_2,E_1)-\dim_k \Ext^2_A(E_2,E_1)=-\langle \hh_2,\hh_1 \rangle_A=2$. Finally, consider exact sequence $0\to E_1\to M\to E_2\to 0$ and the induced exact sequence:
$$
0=\Ext^1_A(E_2,E_2)\to \Ext^2_A(E_2,E_1)\to \Ext^2_A(E_2,M)=0.$$
It is now clear that $\mathcal{E}$ has indeed the desired properties.
\end{proof}

Now, we are ready to prove out theorem:

\begin{proof}[Proof of Theorem \ref{main:thm}] It follows immediately from Proposition \ref{exceptional-seq-tubular-prop}, Proposition \ref{exceptional-Zwara-prop}, and Theorem \ref{Zwara-thm}.
\end{proof}

\subsection*{Acknowledgment} I am very grateful to the referee for a careful reading of the paper. The author was partially supported by NSF grant DMS-1101383.

\end{document}